\documentclass[11pt,a4paper]{amsart}

\usepackage{amsmath,amsthm,amsfonts,mathrsfs}

\newtheorem{theorem}{Theorem}
\newtheorem{lemma}{Lemma}
\newtheorem{corollary}{Corollary}
\newtheorem{proposition}{Proposition}
\theoremstyle{remark}
\newtheorem{remark}{Remark}

\newcommand{\R}{\mathbb{R}}

\allowdisplaybreaks[3]

\begin{document}

\title[A Frostman type lemma for sets with large intersections]{A
  Frostman type lemma for sets with large intersections, and an
  application to Diophantine approximation}

\author{Tomas Persson} \address{Tomas Persson\\ Centre for
  Mathematical Sciences\\ Lund University\\ Box 118\\ 22100
  Lund\\ Sweden} \email{tomasp@maths.lth.se}

\author{Henry W. J. Reeve} \address{Henry W. J. Reeve\\Department of
  Mathematics\\ The University of Bristol\\ University
  Walk\\Clifton\\ Bristol\\BS8 1TW\\UK} \email{henrywjreeve@gmail.com}

\begin{abstract}
  We consider classes $\mathscr{G}^s ([0,1])$ of subsets of $[0,1]$,
  originally introduced by Falconer, that are closed under countable
  intersections, and such that every set in the class has Hausdorff
  dimension at least $s$. We provide a Frostman type lemma to
  determine if a limsup-set is in such a class. Suppose $E = \limsup
  E_n \subset [0,1]$, and that $\mu_n$ are probability measures with
  support in $E_n$. If there is a constant $C$ such that
  \[
  \iint |x - y|^{-s} \, \mathrm{d}\mu_n (x) \mathrm{d} \mu_n (y) < C
  \]
  for all $n$, then under suitable conditions on the limit measure of
  the sequence $(\mu_n)$, we prove that the set $E$ is in the class
  $\mathscr{G}^s ([0,1])$.

  As an application we prove that for $\alpha > 1$ and almost all
  $\lambda \in (\frac{1}{2}, 1)$ the set
  \[
  E_\lambda (\alpha) = \{\, x \in [0,1] : |x - s_n| < 2^{-\alpha n}
  \text{ infinitely often} \, \}
  \]
  where $s_n \in \{\, (1 - \lambda) \sum_{k=0}^n a_k \lambda^k$ and
  $a_k \in \{0, 1\} \,\}$, belongs to the class $\mathscr{G}^s$ for $s
  \leq \frac{1}{\alpha}$. This improves one of our previous results in
  \cite{persson-reeve}.
\end{abstract}


\maketitle

\section{Introduction and Results}

\subsection{Intersection classes}

Falconer introduced in \cite{falconerold} classes $\mathscr{G}^s$ of
subsets of $\R^n$ with the property that any set in $\mathscr{G}^s$
has Hausdorff dimension at least $s$, and countable intersections of
bi-Lipschitz images of sets from $\mathscr{G}^s$, are in
$\mathscr{G}^s$. There are several equivalent ways to characterise the
sets in $\mathscr{G}^s$, see \cite{falconernew}. We will use below a
variant from Bugeaud \cite{bugeaud1}. (There is a minor mistake in the
corresponding part in \cite{falconernew}.)

We define the set functions $\mathscr{M}_\infty^t$ on arbitrary sets
$E \subset \R^n$ as
\[
\mathscr{M}_\infty^t (E) = \inf \biggl\{ \, \sum_i |D_i|^t : E \subset
\bigcup_i D_i \, \biggr\},
\]
where each $D_i$ is a dyadic hypercube. According to \cite{bugeaud1},
$\mathscr{G}^s$ is the class of $G_\delta$ sets $E$ such that for each
$t < s$, there is a constant $c$ such that
\begin{equation} \label{eq:defM}
  \mathscr{M}_\infty^t (E \cap D) \geq c \mathscr{M}_\infty^t (D)
\end{equation}
holds for all dyadic cubes $D$.

If $E_n$ are open sets and $E = \limsup E_n$, then \eqref{eq:defM}
holds and $E$ is in $\mathscr{G}^s$, provided that
\begin{equation} \label{eq:liminfM}
\liminf_{n \to \infty} \mathscr{M}_\infty^t (E_n \cap D) \geq c |D|^t
\end{equation}
holds for all dyadic cubes $D$, see \cite{falconernew}. We will use a
small variation of this result, as stated in
Section~\ref{sec:intersectionclasslemma}.

In this paper, we will consider subsets of the interval $[0,1]$. Since
no subset of $[0,1]$ belongs to the class $\mathscr{G}^s$, we
introduce instead the class $\mathscr{G}^s ([0,1])$, which is the
analog of $\mathscr{G}^s$ for subsets of $[0,1]$. The class
$\mathscr{G}^s ([0,1])$ is the class of $G_\delta$ subsets $E$ of
$[0,1]$ such that if we deploy copies of $E$, translated by an
integer, along the real line, then we get a set that belongs to the
class $\mathscr{G}^s$. Equivalently, $\mathscr{G}^s ([0,1])$ can be
defined as $\mathscr{G}^s$, using \eqref{eq:defM}, where we instead
require that \eqref{eq:defM} holds for all dyadic cubes $D$ that are
subsets of $[0,1]$. With the same change, \eqref{eq:liminfM} can be
used to determine if a set belongs to the class $\mathscr{G}^s
([0,1])$.

Our first result is the following theorem, that gives a method to
determine if a limsup-set belongs to the class $\mathscr{G}^s
([0,1])$.

\begin{theorem} \label{the:frostman}
  Let $E_k$ be open subsets of $[0,1]$, and $\mu_k$ Borel probability
  measures, with support in the closure of $E_k$, that converge weakly
  to a measure $\mu$ with density $h$ in $L^2$. Assume that $\mu (I) >
  0$ for all intervals $I \subset [0, 1]$ with non-empty interior, and
  assume that for each $\varepsilon > 0$, there is a constant
  $C_\varepsilon$, such that
  \begin{equation} \label{eq:est_on_limitmeasure}
    |I|^{1+\varepsilon} \lVert h \chi_I \rVert_2^2 \leq C_\varepsilon
    \lVert h \chi_I \rVert_1^2
  \end{equation}
  holds for any interval $I \subset [0,1]$. If there is a constant $C$
  such that
  \begin{equation} \label{eq:bounded_energy}
    \iint |x - y|^{-s} \, \mathrm{d} \mu_k (x) \mathrm{d} \mu_k (y)
    \leq C
  \end{equation}
  holds for all $k$, then $\limsup E_k$ is in the class
  $\mathscr{G}^s ([0,1])$.
\end{theorem}

We will prove Theorem~\ref{the:frostman} in
Section~\ref{sec:proof_frostman}. Next, we will present our
application of this theorem.

\subsection{Diophantine approximation with $\boldsymbol{\lambda}$-expansions}

Let $\lambda \in (\frac{1}{2}, 1)$, and $\alpha > 1$. We consider the
sets
\[
E_\lambda (\alpha) = \{\, x \in [0,1] : |x - s_n| < 2^{-\alpha n}
\text{ for some } s_n \text{ infinitely often} \, \}
\]
where $s_n \in \{\, (1 - \lambda) \sum_{k=0}^n a_k \lambda^k$ and $a_k
\in \{0, 1\} \, \}$. This set can be written as a limsup-set,
$E_\lambda (\alpha) = \limsup_{n \to \infty} E_{\lambda,n} (\alpha)$,
where
\begin{align*}
  E_{\lambda, n} (\alpha) &= \{\, x \in [0,1] : |x - y| < 2^{-\alpha
    n}
  \text{ for some } y \in F_{\lambda, n} \, \}, \\
  F_{\lambda, n} &= \{\, y : y = (1-\lambda) \sum_{k=0}^n a_k
  \lambda^k, \ a_k \in \{0,1\} \, \}.
\end{align*}

The membership in the classes $\mathscr{G}^s ([0,1])$ of the set
$E_{\lambda} (\alpha)$ for typical $\lambda$, was studied in our paper
\cite{persson-reeve}, where it was proved that $E_\lambda (\alpha)$
belongs to $\mathscr{G}^{1/\alpha} ([0,1])$ for almost all $\lambda
\in (\frac{1}{2}, \frac{2}{3})$. Here, we can state the following
improvement of this result.

\begin{theorem} \label{the:bernoulli}
  For almost all $\lambda \in (\frac{1}{2}, 1)$, the set $E_\lambda
  (\alpha)$ is in $\mathscr{G}^{1/\alpha} ([0,1])$.
\end{theorem}

\begin{remark}
  We note that we cannot have $E_\lambda (\alpha) \in \mathscr{G}^s
  ([0,1])$ for any $s > 1/\alpha$, since a simple covering argument
  shows that the Hausdorff dimension of $E_\lambda (\alpha)$ is not
  larger than $1/\alpha$.

  We should also remark that in our paper \cite{persson-reeve}, we
  studied a different scaling of the sets $E_\lambda (\alpha)$, so
  that they had diameter $\lambda/(1-\lambda)$. This is unimportant
  for the result. In this paper it will prove more convenient to work
  with the sets $E_\lambda (\alpha)$ if they are all subsets of
  $[0,1]$, hence the difference.
\end{remark}

The proof of Theorem~\ref{the:bernoulli} is in
Section~\ref{sec:proof_bernoulli}. It is an application of
Theorem~\ref{the:frostman}.

\section{Proof of Theorem~\ref{the:frostman}} \label{sec:proof_frostman}

\subsection{A Lemma on Large Intersection Classes} \label{sec:intersectionclasslemma}

We start with the following lemma, that will be used later in the
proof. It is the previously mentioned variation of \eqref{eq:liminfM}.

\begin{lemma} \label{lem:intersection}
  Let $E_n$ be open sets and $E = \limsup E_n$.  If for any
  $\varepsilon > 0$ and $t < s$ there is a constant
  $c_{t,\varepsilon}$ such that
  \[
  \liminf_{n \to \infty} \mathscr{M}_\infty^t (E_n \cap D) \geq
  c_{t,\varepsilon} |D|^{t+\varepsilon}
  \]
  holds for all dyadic cubes $D \subset [0,1]$, then $E$ is in the
  class $\mathscr{G}^s ([0,1])$.
\end{lemma}

The proof is a minor perturbation of the proof of Lemma~2 in
\cite{falconernew}.

\begin{proof}
  Let $0 < t < u < s$ and $\varepsilon > 0$.  We take a dyadic cube
  $D$ of length $2^{-m}$, and choose a number $n \geq m$ such that
  \[
  2^{-n (t-u)} \geq c_{u,\varepsilon}^{-1} 2^{-m (t - u -
    \varepsilon)}.
  \]

  Let $\{I_i\}$ be any disjoint cover of $E \cap D$ by dyadic
  cubes. We write $D$ as a finite union of disjoint dyadic cubes,
  \[
  D = \bigcup_{j=1}^k J_j,
  \]
  such that for any $j$ either
  \begin{itemize}
  \item[i)] $J_j = I_i$ for some $i$ and $|J_j| > 2^{-n}$,
  \end{itemize}
  or
  \begin{itemize}
  \item[ii)] $|J_j| = 2^{-n}$ and those $I_i$ that cover $E \cap J_j$
    are subsets of $J_j$.
  \end{itemize}
  
  Let $Q (j) = \{\, i : I_i \subseteq J_j \,\}$. If $j$ satisfies i),
  then $Q (j)$ has exactly one element, and so
  \begin{equation} \label{eq:intersectionclass}
    \sum_{i \in Q (j)} |I_i|^t = |J_j|^t = |J_j|^{t-u} |J_j|^u \geq
    |D|^{t-u} |J_j|^u \geq |D|^{t-u-\varepsilon}
    |J_j|^{u+\varepsilon}.
  \end{equation}

  If $j$ satisfies ii), then for $i \in Q (j)$ we have
  \[
  |I_i| = |I_i|^{t-u} |I_i|^u \geq 2^{-n (t-u)} |I_i|^u \geq
  c_{u,\varepsilon}^{-1} 2^{-m (t - u - \varepsilon)} |I_i|^u =
  c_{u,\varepsilon}^{-1} |D|^{t - u - \varepsilon} |I_i|^u.
  \]
  Hence, summing over $i \in Q(j)$, we get
  \begin{align}
    \nonumber \sum_{i \in Q (j)} |I_i|^t &\geq c_{u,\varepsilon}^{-1}
    |D|^{t - u - \varepsilon} \sum_{i \in Q (j)} |I_i|^u \\ &\geq
    c_{u,\varepsilon}^{-1} |D|^{t - u - \varepsilon}
    \mathscr{M}_\infty^u (E \cap J_j) \geq |D|^{t - u - \varepsilon}
    |J_j|^{u + \varepsilon}. \label{eq:intersectionclass2}
  \end{align}

  Combining \eqref{eq:intersectionclass} and
  \eqref{eq:intersectionclass2} we get
  \begin{align*}
    \sum_{i=1}^\infty |I_i|^t &\geq |D|^{t - u - \varepsilon}
    \sum_{j=1}^k |J_j|^{u + \varepsilon} \\ &\geq |D|^{t - u -
      \varepsilon} \mathscr{M}_\infty^{u + \varepsilon} (D) = |D|^{t -
      u - \varepsilon} |D|^{u + \varepsilon} = |D|^u.
  \end{align*}
  This shows that $E \in \mathscr{G}^u ([0,1])$. Since $u$ was
  arbitrary, the conclusion of the lemma follows.
\end{proof}

\subsection{Some Notation}

We will work with functions and probability measures on the interval
$[0,1]$. For a function $f \colon [0,1] \to \R$ and $0 < s < 1$, we
denote by $R_s f$ the function
\[
R_s f (x) = \int |x - y|^{-s} f (y) \, \mathrm{d}y,
\]
provided that the integral exists.  Similarly, if $\mu$ is a measure
on $[0,1]$, we let
\[
R_s \mu (x) = \int |x - y|^{-s} \, \mathrm{d} \mu (y),
\]
provided that the integral exists. This is the case, for instance, if
$f$ and the density of $\mu$ are in $L^p$ for some $p >
\frac{1}{1-s}$, since then Young's inequality implies, with $p >
\frac{1}{1-s}$ and $q < \frac{1}{s}$ such that $1/p + 1/q \leq 1$,
that
\[
\lVert R_s f \rVert_\infty \leq \lVert |\cdot|^{-t} \rVert_q \lVert f
\rVert_p < \infty.
\]

\subsection{Some lemmata}

In this section, we prove some basic estimates that will be used in
the proof of Theorem~\ref{the:frostman}.

\begin{lemma} \label{lem:Mm}
  Let $\mu$ be a Borel measure on $[0,1]$. Assume that for some $s>0$
  holds
  \[
  C = \iint |x-y|^{-s} \, \mathrm{d}\mu (x) \mathrm{d}\mu (y) <
  \infty.
  \]
  Then, if $M_m = \{\, (x,y) : |x-y|^{-s} > m \,\}$, we have for $0 <
  t <s$
  \[
  \iint_{M_m} |x-y|^{-t} \, \mathrm{d} \mu (x) \mathrm{d} \mu (y) \leq
  C \frac{s}{s-t} m^{t/s-1}.
  \]
\end{lemma}

\begin{proof}
  We note that $\iint |x-y|^{-t} \, \mathrm{d} \mu (x) \mathrm{d} \mu
  (y) < C$ for $t < s$ and that $M_m = \{\, (x,y) : |x-y|^{-t} >
  m^{t/s} \, \}$. Using that $\mu \times \mu (M_m) \leq C/m$ we get
  \begin{align*}
    \iint_{M_m} |x-y|^{-t} \, \mathrm{d} \mu (x) \mathrm{d} \mu (y) &=
    \int_{m^{t/s}}^\infty \mu \times \mu (M_{u^{s/t}}) \, \mathrm{d} u
    + m^{t/s} \mu \times \mu (M_m) \\ &\leq \int_{m^{t/s}}^\infty
    \frac{C}{u^{s/t}} \, \mathrm{d} u + C m^{t/s - 1} \\ &= C
    \frac{s}{s-t} m^{t/s - 1}. \qedhere
  \end{align*}
\end{proof}

\begin{corollary} \label{cor:converge}
  If $\mu_n$ are Borel measures on $[0,1]$ that converge weakly to a
  measure $\mu$, and $\iint |x-y|^{-s} \, \mathrm{d}\mu_n (x)
  \mathrm{d} \mu_n (y)$ are uniformly bounded for some $s > 0$, then
  for $t < s$
  \[
  \iint |x-y|^{-t} \, \mathrm{d}\mu_n (x) \mathrm{d} \mu_n (y) \to
  \iint |x-y|^{-t} \, \mathrm{d}\mu (x) \mathrm{d} \mu (y)
  \]
  as $n \to \infty$.
\end{corollary}

\begin{proof}
  Let $\varepsilon > 0$ and $t < s$. Let $M_m$ denote the same set as
  before. By Lemma~\ref{lem:Mm} we can take $m$ and $N$ so large that
  \[
  \iint_{M_m} |x - y|^{-t} \, \mathrm{d} \mu_n (x) \mathrm{d} \mu_n
  (y) < \varepsilon
  \]
  for all $n > N$. Then
  \[
  \iint |x - y|^{-t} \, \mathrm{d} \mu_n (x) \mathrm{d} \mu_n (y) \leq
  \varepsilon + \iint \min \{ |x - y|^{-t}, m^{t/s} \} \, \mathrm{d}
  \mu_n (x) \mathrm{d} \mu_n (y).
  \]
  The function $\phi_m \colon (x,y) \mapsto \min \{ |x - y|^{-t},
  m^{t/s} \}$ is continuous, and so
  \[
  \iint \phi_m \, \mathrm{d} \mu_n \mathrm{d} \mu_n \to \iint \phi_m
  \, \mathrm{d} \mu \mathrm{d} \mu \leq \iint |x - y|^{-t} \,
  \mathrm{d} \mu (x) \mathrm{d} \mu (y),
  \]
  as $n \to \infty$. Since $\varepsilon$ is arbitrary this shows that
  \[
  \limsup_{n\to\infty} \iint |x - y|^{-t} \, \mathrm{d} \mu_n (x)
  \mathrm{d} \mu_n (y) \leq \iint |x - y|^{-t} \, \mathrm{d} \mu (x)
  \mathrm{d} \mu (y).
  \]
  Now, the (obvious) inequality
  \[
  \liminf_{n\to\infty} \iint |x - y|^{-t} \, \mathrm{d} \mu_n (x)
  \mathrm{d} \mu_n (y) \geq \iint |x - y|^{-t} \, \mathrm{d} \mu (x)
  \mathrm{d} \mu (y)
  \]
  proves the corollary.
\end{proof}

\begin{lemma} \label{lem:bounded}
  Suppose $h \colon [0,1] \to \R$ is a non-negative function such
  that $R_s h$ is bounded. If $U \subset I \subset [0,1]$ are two
  intervals, then
  \[
  \frac{1}{|U|^s} \int_U \frac{h}{R_s (h|_I)} \, \mathrm{d}x \leq 1
  \]
  provided $h$ does not vanish a.e.\ on $I$.
\end{lemma}

\begin{proof}
  Let $V$ be an interval. We first note that if $x \in V$, then
  \begin{equation} \label{eq:RsI_est}
    \frac{1}{1-s} |V|^{1-s} \leq \int_V |x - y|^{-s} \, \mathrm{d}y
    \leq \frac{2^s}{1-s} |V|^{1-s}.
  \end{equation}
  Hence we have
  \begin{equation} \label{eq:RsI_est2}
    |V|^{1-s} \chi_V (x) \leq \frac{|V|^{1-s}}{1-s} \chi_V (x) \leq
    \int |x - y|^{-s} \chi_V (y) \, \mathrm{d}y = R_s \chi_V (x)
  \end{equation}
  for any $x$. For $x \not \in V$ we have
  \begin{equation} \label{eq:RsI_est3}
  R_s \chi_V (x) \geq d^{-s} |V|,
  \end{equation}
  where $d = \sup \{\, |x-y| : y \in V \, \}$.
  
  Assume that $h|_I$ is of the form
  \begin{equation} \label{eq:hstep}
    h|_I = \sum_{k \in J} c_k \chi_{I_k},
  \end{equation}
  where $(I_k)$ are disjoint intervals that are subsets of $I$. Let
  $J_U \subset J$ be the set of indices such that $I_k$ is a subset of
  $U$ for $k \in J_U$. Then
  \[
  R_s (h|_I) \geq \sum_{k \in J_U} c_k R_s \chi_{I_k}.
  \]
  By \eqref{eq:RsI_est2} and \eqref{eq:RsI_est3}, we have for $x \in
  I_l$ that
  \[
  R_s (h|_I) (x) \geq c_l |I_l|^{1-s} + \sum_{J_U \ni k \neq l} c_k
  |U|^{-s} |I_k|,
  \]
  and
  \[
  \frac{h (x)}{R_s (h|_I) (x)} \leq \frac{c_l}{\displaystyle c_l
    |I_l|^{1-s} + \sum_{J_U \ni k \neq l} c_k |U|^{-s} |I_k|}.
  \]
  Hence
  \[
  \frac{1}{|U|^s} \int_U \frac{h(x)}{R_s (h|_I) (x)} \, \mathrm{d} x
  \leq |U|^{-s} \sum_{l \in J_U} \frac{c_l |I_l|}{\displaystyle c_l
    |I_l|^{1-s} + \sum_{J_U \ni k \neq l} c_k |U|^{-s} |I_k|}.
  \]
  If we assume that $|I_k| = d$ for all $k$, then
  \begin{align*}
    \frac{1}{|U|^s} \int_U \frac{h(x)}{R_s (h|_I) (x)} \, \mathrm{d} x
    & \leq |U|^{-s} \sum_{l \in J_U} \frac{c_l d}{\displaystyle c_l
      d^{1-s} + \sum_{J_U \ni k \neq l} c_k |U|^{-s} d} \\ & = \sum_{l
      \in J_U} \frac{c_l |U|^{-s} d^s}{\displaystyle c_l + \sum_{J_U
        \ni k \neq l} c_k |U|^{-s} d^s} \\ & \leq \sum_{l \in J_U}
    \frac{c_l |U|^{-s} d^s}{\displaystyle c_l |U|^{-s} d^s + \sum_{J_U
        \ni k \neq l} c_k |U|^{-s} d^s} \leq 1.
  \end{align*}

  The general case is now proved by approximating with $h$ of the form
  \eqref{eq:hstep}, with $|I_k| = |I_l|$.
\end{proof}

\subsection{Final part of the proof of Theorem~\ref{the:frostman}}

We will prove that for any $t < s$ and $\varepsilon > 0$, there is a
constant $c$ such that
\begin{equation} \label{eq:Minfty}
  \liminf \mathscr{M}_\infty^t (E_n \cap I) \geq c |I|^{t +
    \varepsilon}
\end{equation}
holds for any interval $I \subset [0,1]$. This implies according to
Lemma~\ref{lem:intersection}, that $E$ is in $\mathscr{G}^s ([0,1])$.

Let $I \subset [0,1]$ be fixed and fix $t < s$. We denote by
$\mu_n|_I$ the restriction of $\mu_n$ to $I$, i.e.\ $\mu_n|_I (A) =
\mu_n (I \cap A)$. For large enough $n$ we may assume that $\mu_n (I)
> 0$.

We define new measures $\nu_n$ by
\[
\nu_n (A) = \frac{\int_A (R_t \mu_n|_I)^{-1} \, \mathrm{d}
  \mu_n}{\int_I (R_t \mu_n|_I)^{-1} \, \mathrm{d} \mu_n}.
\]
Clearly, the support of $\nu_n$ is equal to the support of $\mu_n|_I$.

We will prove that if $n$ is large enough, then $\nu_n$ satisfies the
estimates
\begin{equation} \label{eq:nu_est}
  \nu_n (U) \leq c \frac{|U|^t}{|I|^{t+\varepsilon}}
\end{equation}
for each interval $U \subset I$, where $c$ does not depend on $I$.

Suppose we have \eqref{eq:nu_est}. Let $\{U_k\}$ be a cover of $I \cap
E_n$ by disjoint intervals. Then, since the support of $\nu_n$ lies
inside the closure of $I \cap E_n$, we have
\[
1 = \nu_n (\cup U_k) = \sum \nu_n (U_k) \leq \sum c
\frac{|U_k|^t}{|I|^{t+\varepsilon}},
\]
which implies that
\[
\sum |U_k|^t \geq c^{-1} |I|^{t+\varepsilon}
\]
holds for any cover $\{U_k\}$. This implies \eqref{eq:Minfty}.

It remains to prove that \eqref{eq:nu_est} holds for large enough
$n$. First, we see that \eqref{eq:nu_est} is equivalent to the
inequality
\begin{equation} \label{eq:int_est}
  \frac{1}{|U|^t} \int_U (R_t \mu_n|_I)^{-1} \, \mathrm{d} \mu_n \leq
  \frac{c}{|I|^{t+\varepsilon}} \int_I (R_t \mu_n|_I)^{-1} \,
  \mathrm{d} \mu_n.
\end{equation}
Consider the left side of \eqref{eq:int_est}. By
Lemma~\ref{lem:bounded} this is bounded from above by the constant
$1$, which is a constant that is independent of $U$, $I$ and $n$.

The right side of \eqref{eq:int_est} is estimated as
\[
\int_I (R_t \mu_n|_I)^{-1} \, \frac{\mathrm{d} \mu_n}{\mu_n (I)} \geq
\biggl( \int_I R_t \mu_n|_I \, \frac{\mathrm{d} \mu_n}{\mu_n (I)}
\biggr)^{-1} = \frac{ \mu_n (I)}{ \int_I R_t \mu_n|_I \, \mathrm{d}
  \mu_n }.
\]
Hence
\begin{equation} \label{eq:lower_est}
  \int_I (R_t \mu_n|_I)^{-1} \, \mathrm{d} \mu_n \geq \frac{ \mu_n
    (I)^2}{ \int_I R_t \mu_n|_I \, \mathrm{d} \mu_n }.
\end{equation}

By Corollary~\ref{cor:converge} we have for $t < s$ that
\[
\int_I R_t \mu_n|_I \, \mathrm{d} \mu_n \to \int_I R_t \mu|_I \,
\mathrm{d} \mu,
\]
as $n \to \infty$. Hence the right hand side of \eqref{eq:lower_est}
converges to
\[
\frac{ \mu (I)^2}{ \int_I R_t \mu |_I \, \mathrm{d} \mu }.
\]
We want to prove that there exists a constant $c$, such that
\[
  \frac{c}{|I|^{t+\varepsilon}} \int_I (R_t \mu_n|_I)^{-1} \,
  \mathrm{d} \mu_n \geq 1,
\]
holds for large enough $n$. To do so, it is sufficient to prove that
\begin{equation} \label{eq:ineq_for_c}
  \frac{c}{|I|^{t+\varepsilon}} \frac{ \mu (I)^2}{ \int_I R_t \mu |_I
    \, \mathrm{d} \mu } \geq 1 \geq \frac{1}{|U|} \int_U (R_t
  \mu_n|_I)^{-1}\, \mathrm{d}\mu_n,
\end{equation}
where the rightmost inequality has been obtained above. The leftmost
inequality is proved using \eqref{eq:est_on_limitmeasure} as
follows. Let $g (u) = |u|^{-t}$ for $0 < |u| \leq |I|$ and $g(u) = 0$
otherwise. Then
\[
\int_I R_t \mu|_I \, \mathrm{d} \mu = \int_I \int_I g(x-y) h(y) \,
\mathrm{d}y \, h(y) \, \mathrm{d}x = \int (g *(h \chi_I)) (h \chi_I)
\, \mathrm{d}x
\]
Using H\"older's inequality and then Young's inequality we get
\[
\int_I R_t \mu|_I \, \mathrm{d} \mu \leq \lVert g * (h \chi_I)
\rVert_2 \lVert h \chi_I \rVert_2 \leq \lVert g \rVert_1 \lVert h
\chi_I \rVert_2^2.
\]
Since $\lVert g \rVert_1 = \frac{2}{1-t} |I|^{1-t}$,
\eqref{eq:est_on_limitmeasure} implies that
\[
\int_I R_t \mu|_I \, \mathrm{d} \mu \leq \frac{2 C_\varepsilon}{1-t}
|I|^{-t-\varepsilon} \lVert h \chi_I \rVert_1^2 = \frac{2
  C_\varepsilon}{1-t} |I|^{-t-\varepsilon} \mu (I)^2.
\]
It is now apparent that \eqref{eq:ineq_for_c} holds if we choose $c >
2C_\varepsilon / (1 - t)$.

This establishes \eqref{eq:int_est}, and hence finishes the proof.

\section{Proof of Theorem~\ref{the:bernoulli}} \label{sec:proof_bernoulli}

Here we will prove Theorem~\ref{the:bernoulli}. The proof, that is
based on Theorem~\ref{the:frostman}, is divided into three parts,
found in Sections~\ref{sec:densities},
\ref{sec:fourier}--\ref{sec:fourierproof} and
\ref{sec:convolution}. We will construct measures that for almost all
$\lambda$ satisfies the assumptions of Theorem~\ref{the:frostman}. In
Section~\ref{sec:densities}, we prove that the assumption
\eqref{eq:est_on_limitmeasure} of Theorem~\ref{the:frostman} is
satisfied, and in Sections~\ref{sec:fourier} and
\ref{sec:fourierproof} we prove that the assumption
\eqref{eq:bounded_energy} of Theorem~\ref{the:frostman} is satisfied
for almost all $\lambda \in (\frac{1}{2}, 0.64)$. Finally, in
Section~\ref{sec:convolution}, we show how to conclude the desired
result for almost all $\lambda \in (\frac{1}{2}, 1)$.

\subsection{Some estimates on densities} \label{sec:densities}

Let us now begin the proof of Theorem~\ref{the:bernoulli}.  According
to Theorem~\ref{the:frostman}, to prove that $E_\lambda (\alpha)$ is
in $\mathscr{G}^s ([0,1])$, it is sufficient to construct probability
measures $\mu_{\lambda,k}$ with support in $E_{\lambda,k} (\alpha)$,
converging weakly to a measure $\mu_\lambda$ with density $h_\lambda$
in $L^2$, such that there exist constants $C_\varepsilon$ and $C$ with
the property that
\begin{equation*}
  |I|^{1 + \varepsilon} \lVert h_\lambda \chi_I \rVert_2^2 \leq
  C_\varepsilon \lVert h_\lambda \chi_I \rVert_1^2
\end{equation*}
holds for all intervals $I \subset [0,1]$, and
\begin{equation*}
  \iint |x-y|^{-s} \, \mathrm{d}\mu_{\lambda,k} (x) \mathrm{d}
  \mu_{\lambda,k} (y) \leq C
\end{equation*}
holds for infinitely many $k$. We will construct such measures for all
$\lambda \in (\frac{1}{2}, 1)$, and prove that constants
$C_\varepsilon$ and $C$ with the properties mentioned above, exist
for almost all $\lambda \in (\frac{1}{2}, 1)$.

The measures $\mu_{\lambda,k}$ are constructed in the following
way. We put
\[
\Sigma_k = \{\, (a_0, a_1, \ldots, a_k) : a_n \in \{0,1\} \,\},
\]
and define $\pi_k \colon \Sigma_k \to [0,1]$ by $\pi_k \colon (a_0,
a_1, \ldots, a_k) \mapsto (1 - \lambda) \sum_{n=0}^k a_n
\lambda^n$. Hence $F_{\lambda, k} = \pi_k (\Sigma_k)$. Let $\nu$
denote the Lebesgue measure and let $\nu_I$ denote the normalised
Lebesgue measure on an interval $I$. We denote by $B_r (x)$ the closed
interval of length $2r$ and centre at $x$. Put
\begin{equation} \label{eq:mudef}
  \mu_{\lambda,k} = 2^{-k} \sum_{a \in \Sigma_k} \nu_{B_{r_k}
    (\pi_k (a) )},
\end{equation}
where $r_k = 2^{-\alpha k}$. Then $\mu_{\lambda,k}$ is a probability
measure with support $E_{\lambda, k} (\alpha)$, and $\mu_{\lambda,k}$
converges weakly to a measure $\mu_\lambda$ as $k \to \infty$. The
measure $\mu_\lambda$ is the distribution of the random Bernoulli
convolution as described in \cite{solomyak}, where it is proved that
$\mu_{\lambda}$ has a density $h_\lambda$ in $L^2$ for almost all
$\lambda \in (\frac{1}{2}, 1)$. It gives positive measure to any
interval $I \subset [0, 1]$ with non-empty interior.

Let $\lambda$ be such that $h_\lambda$ has density in $L^2$. The
density $h_\lambda$ satisfies the functional equation
\begin{equation} \label{eq:functionalequation}
  h_\lambda = \frac{1}{2 \lambda} h_\lambda \circ S_1^{-1} +
  \frac{1}{2 \lambda} h_\lambda\circ S_2^{-1},
\end{equation}
where $S_1$ and $S_2$ are the two contractions
\begin{align*}
  S_1 \colon x &\mapsto \lambda x \\
  S_2 \colon x &\mapsto \lambda x + 1 - \lambda.
\end{align*}
This can also be written in the following form. If $I$ is an interval,
and $I_1$, $I_2$ are two intervals such that $S_1 (I_1) = I$ and $S_2
(I_2) = I$, then
\begin{align*}
  \int_I h_\lambda \, \mathrm{d} \nu &= \frac{1}{2 \lambda}
  \int_I h_\lambda \circ S_1^{-1} \, \mathrm{d} \nu + \frac{1}{2
    \lambda} \int_I h_\lambda \circ S_2^{-1} \, \mathrm{d} \nu \\ &=
  \frac{1}{2} \int_{I_1} h_\lambda \, \mathrm{d} \nu + \frac{1}{2}
  \int_{I_2} h_\lambda \, \mathrm{d} \nu,
\end{align*}
or equivalently
\begin{equation} \label{eq:functionalequation_measureform}
  \mu_\lambda (I) = \frac{1}{2} \mu_\lambda (S_1^{-1} (I)) +
  \frac{1}{2} \mu_\lambda (S_2^{-1} (I)).
\end{equation}

We prove the following property of the measure $\mu_\lambda$.

\begin{proposition} \label{pro:est_on_bernoulli}
  If $\lambda$ is such that $\mu_\lambda$ has density $h_\lambda$ in
  $L^2$, then for any $\varepsilon > 0$, there exists a constant
  $C_\varepsilon$ such that
  \begin{equation} \label{eq:est_on_bernoulli} |I|^{1 + \varepsilon}
    \lVert h_\lambda \chi_I \rVert_2^2 \leq C_\varepsilon \lVert
    h_\lambda \chi_I \rVert_1^2
  \end{equation}
  holds for any interval $I \subset [0,1]$.
\end{proposition}

To prove Proposition~\ref{pro:est_on_bernoulli} we will need the
following lemma.

\begin{lemma} \label{lem:estimates_on_measure}
  Put $\theta = - \frac{\log 2}{\log \lambda}$ and fix $\lambda \in
  (\frac{1}{2},1)$. Then there is a constant $K$ such that
  \[
  \frac{K}{2} (1 - \lambda)^{-\theta} r^\theta \leq \mu_\lambda
  ([0,r)) \leq 2 K (1 - \lambda)^{-\theta} r^\theta.
  \]
  Moreover, there is a constant $c$ such that for any interval $I
  \subset [0,1]$ of length $r$, holds
  \[
  \mu_\lambda (I) \geq c r^\theta.
  \]
\end{lemma}

\begin{proof}
  Let $V_0 = [0, 1 - \lambda)$. Then $[0,1] \cap S_2^{-1} (V_0) =
    \emptyset$ and by \eqref{eq:functionalequation_measureform}, we
    have $K = \mu_\lambda (V_0)$, for some constant $K$.

  Now, let $V_k$ be defined recursively by $V_k = S_1 (V_{k-1}) = [0,
  \lambda^k (1 - \lambda))$. Then $\mu_\lambda (V_k) = 2^{-k}
  \mu_\lambda (V_0)$, so
  \[
  \mu_\lambda (V_k) = K 2^{-k}.
  \]
  Consider the interval $[0,r)$, were $r \leq 1 - \lambda$. We let $n$
    be an integer such that
  \[
  V_n \subset [0,r) \subset V_{n-1}.
  \]
  Then 
  \[
  \frac{\log r - \log (1- \lambda)}{\log \lambda} \leq n \leq 1 +
  \frac{\log r - \log (1- \lambda)}{\log \lambda}.
  \]
  We have 
  \[
  \mu_\lambda ([0,r)) \geq \mu_\lambda (V_n) = K 2^{-n} \geq
    \frac{K}{2} (1 - \lambda)^{-\theta} r^\theta.
  \]
  and
  \[
  \mu_\lambda ([0,r)) \leq \mu_\lambda (V_{n-1}) = K 2^{-n+1} \leq 2
    K (1 - \lambda)^{-\theta} r^\theta.
  \]

  Let $c_0$ be the minimal $\mu_\lambda$-measure of a sub-interval of
  $[0,1]$ of length $2 \lambda - 1$. We have $c_0 > 0$. Consider an
  interval $I \subset [0,1]$. If $|I| < 2 \lambda - 1$, then at least
  one of the intervals $S_1^{-1} (I)$ and $S_2^{-1} (I)$ are
  sub-intervals of $[0,1]$. So if $|I| = \lambda^k (2 \lambda - 1)$
  then there is an interval $I' \subset [0,1]$ of length $2 \lambda -
  1$, such that $I = S_{i_1} \circ \cdots \circ S_{i_k} (I')$. By
  \eqref{eq:functionalequation_measureform} we have $\mu_\lambda (I)
  \geq 2^{-k} c_0$.

  Finally, if $I \subset [0,1]$ is an arbitrary interval of length $r$
  we can choose an interval $J \subset I$ with $|J| = \lambda^{n} (2
  \lambda -1)$, but $\lambda |I| \leq |J|$. Since $\mu_\lambda (I)
  \geq \mu_\lambda (J) \geq 2^{-n} c_0 = c_0 (2 \lambda - 1)^{-\theta}
  |J|^\theta \geq c_0 (2 \lambda - 1)^{-\theta} \lambda^\theta
  r^\theta$, we conclude the theorem with $c = c_0 (2 \lambda -
  1)^{-\theta} \lambda^\theta$.
\end{proof}

\begin{proof}[Proof of Proposition~\ref{pro:est_on_bernoulli}]
  Fix $\rho > 0$ and consider the intervals $I_x = B_\rho (x)$. The
  function
  \[
  f (x) = \frac{\lVert h_\lambda \chi_{I_x} \rVert_2^2}{\lVert
    h_\lambda \chi_{I_x} \rVert_1^2}
  \]
  defined on the interval $[0,1]$, is clearly continuous, and hence it
  is bounded on $[0,1]$. Moreover, using \eqref{eq:functionalequation}
  one can prove that for $x = 0$ and $x = 1$, the function
  \[
  g (r) = \frac{r \lVert h_\lambda \chi_{B_r (x)} \rVert_2^2}{\lVert
    h_\lambda \chi_{B_r (x)} \rVert_1^2}
  \]
  is bounded when $r \to 0$. This is done in the following way. Assume
  that $x = 0$. The case $x = 1$ is similar by symmetry.  Let $r$ be
  fixed. We are going to estimate $g (\lambda r)$ in terms of $g
  (r)$. Let $J = [0, \lambda r]$.  There are two intervals $J_1$ and
  $J_2$ such that $J = S_1 (J_1) = S_2 (J_2)$. We clearly have $J_1 =
  [0, r]$, and if $r$ is sufficiently small, then $J_2 \cap [0, 1] =
  \emptyset$. Therefore, we assume that $r$ is so small that we have
  $J_2 \cap [0, 1] = \emptyset$.

  Now, $\int_{J_2} h_\lambda \, \mathrm{d} \nu = 0$ and
  \eqref{eq:functionalequation_measureform} implies that
  \begin{equation}
    \int_J h_\lambda \, \mathrm{d} \nu = \frac{1}{2} \int_{J_1}
    h_\lambda \, \mathrm{d} \nu + \frac{1}{2} \int_{J_2} h_\lambda \,
    \mathrm{d} \nu = \frac{1}{2} \int_{J_1} h_\lambda \, \mathrm{d}
    \nu,
    \label{eq:g_estimate}
  \end{equation}
  and by \eqref{eq:functionalequation},
  \begin{equation}
    \int_J h_\lambda^2 \, \mathrm{d} \nu = \frac{1}{4 \lambda^2}
    \int_J \bigl( h_\lambda \circ S_1^{-1} + h_\lambda \circ S_2^{-1}
    \bigr)^2 \, \mathrm{d} \nu = \frac{1}{4 \lambda} \int_{J_1}
    h_\lambda^2 \, \mathrm{d} \nu. \label{eq:g_estimate2}
  \end{equation}
  By the definition of $g$ we have
  \[
  g (r) = r \frac{\int_{J_1} h_\lambda^2 \, \mathrm{d} \nu}{\bigl(
    \int_{J_1} h_\lambda \, \mathrm{d} \nu \bigr)^2}.
  \]
  Hence, by \eqref{eq:g_estimate} and \eqref{eq:g_estimate2},
  \[
  g(\lambda r) = \lambda r \frac{\int_{J} h_\lambda^2 \, \mathrm{d}
    \nu}{\bigl( \int_{J} h_\lambda \, \mathrm{d} \nu \bigr)^2} =
  \lambda r \frac{\frac{1}{4 \lambda} \int_{J_1} h_\lambda^2 \,
    \mathrm{d} \mu}{\frac{1}{4} \bigl( \int_{J_1} h_\lambda \,
    \mathrm{d} \nu \bigr)^2} = g (r).
  \]
  By induction, we conclude that $g (\lambda^n r) = g(r)$ for all $n >
  0$. It is moreover easy to see that $g$ must be bounded on the
  interval $[\lambda \rho, \rho]$, and so $g$ is bounded on $(0,
  \rho]$. (Indeed, $g$ is continuous on any closed sub-interval of
  $(0,1)$.)

  We have proved that both $f$ and $g$ are bounded.
  Let $C_0$ be a constant such that
  $f \leq C_0 / (2\rho) = C_0 / |I_x|$ on $[-\rho, 1 + \rho]$. This
  means that we have
  \begin{equation} \label{eq:induction}
    |I| \lVert h_\lambda \chi_{I} \rVert_2^2 \leq C_0 \lVert h_\lambda
    \chi_{I} \rVert_1^2
  \end{equation}
  for all intervals $I \subset [0,1]$ of length $2 \rho$. We also let
  $C_1$ be a constant such that $g \leq C_1$. Hence we have
  \begin{equation} \label{eq:induction2}
    r \lVert h_\lambda \chi_{[0, r)} \rVert_2^2 \leq C_1 \lVert
      h_\lambda \chi_{[0, r)} \rVert_1^2
  \end{equation}
  for all $0 < r < 1$. By symmetry of $\mu_\lambda$, we have the same
  inequality for the intervals $(1-r, 1]$.

  We will now proceed by induction in the following way. Suppose
  \eqref{eq:induction} holds for all intervals $I \subset [0,1]$ of a
  certain length $L < 1$.  Let $J \subset [0,1]$ be an interval with
  $|J| = \lambda L$. We want to prove that
  \[
  |J| \lVert h_\lambda \chi_J \rVert_2^2 \leq C_2 \lVert h_\lambda
  \chi_J \rVert_1^2.
  \]
  There are two intervals $J_1$ and $J_2$ such
  that $J = S_1 (J_1) = S_2 (J_2)$. By \eqref{eq:functionalequation}
  we have
  \begin{equation}
    \int_J h_\lambda \, \mathrm{d} \nu = \frac{1}{2} \int_{J_1}
    h_\lambda \, \mathrm{d} \nu + \frac{1}{2} \int_{J_2} h_\lambda \,
    \mathrm{d} \nu,
    \label{eq:L1}
  \end{equation}
  and
  \begin{align*}
    \int_J h_\lambda^2 \, \mathrm{d} \nu &= \frac{1}{4 \lambda^2}
    \int_J \bigl( h_\lambda \circ S_1^{-1} + h_\lambda \circ S_2^{-1}
    \bigr)^2 \, \mathrm{d} \nu \\ &= \frac{1}{4 \lambda} \int_{J_1}
    h_\lambda^2 \, \mathrm{d} \nu + \frac{1}{4 \lambda} \int_{J_2}
    h_\lambda^2 \, \mathrm{d} \nu + \frac{1}{2 \lambda} \int_{J_1}
    h_\lambda \cdot h_\lambda \circ T \, \mathrm{d} \nu,
  \end{align*}
  where $T = S_2^{-1} \circ S_1$ is a translation. By the
  Cauchy--Bunyakovsky--Schwarz inequality we have
  \[
  \int_{J_1} h_\lambda \cdot h_\lambda \circ T \, \mathrm{d} \nu \leq
  \biggl( \int_{J_1} h_\lambda^2 \, \mathrm{d} \nu \int_{J_2}
  h_\lambda^2 \, \mathrm{d} \nu \biggr)^\frac{1}{2},
  \]
  and so
  \begin{align*}
    \int_J h_\lambda^2 \, \mathrm{d} \nu &\leq \frac{1}{4 \lambda}
    \int_{J_1} h_\lambda^2 \, \mathrm{d} \nu + \frac{1}{4 \lambda}
    \int_{J_2} h_\lambda^2 \, \mathrm{d} \nu + \frac{1}{2 \lambda}
    \biggl( \int_{J_1} h_\lambda^2 \, \mathrm{d} \nu \int_{J_2}
    h_\lambda^2 \, \mathrm{d} \nu \biggr)^\frac{1}{2} \\ &= \frac{1}{4
      \lambda} \biggl( \biggl( \int_{J_1} h_\lambda^2 \, \mathrm{d}
    \nu \biggr)^{\frac{1}{2}} + \biggl( \int_{J_2} h_\lambda^2 \,
    \mathrm{d} \nu \biggr)^{\frac{1}{2}} \biggr)^2.
  \end{align*}
  We want to use \eqref{eq:induction} on each of the integrals
  $\int_{J_1} h_\lambda^2 \, \mathrm{d}\nu$ and $\int_{J_2}
  h_\lambda^2 \, \mathrm{d}\nu$ above, but it may happen that one of
  $J_1$ and $J_2$ is not a subset of $[0,1]$. Assume therefore that
  $J_1 \subset [0,1]$ and that $J_2$ is not necessarily a subset of
  $[0,1]$. (The case $J_2 \subset [0,1]$ and $J_1$ is not necessarily
  a subset of $[0,1]$ is analogous.) Let $\tilde{J}_2$ be the
  intersection $\tilde{J}_2 = J_2 \cap [0,1]$. We have $0 \leq
  |\tilde{J}_2| \leq |J_1|$. Our estimate \eqref{eq:induction2}
  implies that
  \[
  \int_{J_2} h_\lambda^2 \, \mathrm{d} \nu = \int_{\tilde{J}_2}
  h_\lambda^2 \, \mathrm{d} \nu \leq \frac{C_1}{|\tilde{J}_2|} \biggl(
  \int_{\tilde{J}_2} h_\lambda \, \mathrm{d} \nu \biggr)^2 =
  \frac{C_1}{|\tilde{J}_2|} \biggl( \int_{{J}_2} h_\lambda \,
  \mathrm{d} \nu \biggr)^2,
  \]
  and \eqref{eq:induction} implies that
  \[
  \int_{J_1} h_\lambda^2 \, \mathrm{d} \nu \leq
  \frac{C_0}{|J_1|} \biggl( \int_{J_1} h_\lambda \, \mathrm{d}
  \nu \biggr)^2.
  \]
  Hence
  \begin{align}
    \nonumber \int_J h_\lambda^2 \, \mathrm{d} \nu &\leq \frac{1}{4
      \lambda} \biggl( \biggl( \int_{J_1} h_\lambda^2 \, \mathrm{d}
    \nu \biggr)^{\frac{1}{2}} + \biggl( \int_{J_2} h_\lambda^2 \,
    \mathrm{d} \nu \biggr)^{\frac{1}{2}} \biggr)^2 \\ \nonumber &\leq
    \frac{1}{4 \lambda} \biggl(
    \frac{C_0^\frac{1}{2}}{|J_1|^\frac{1}{2}} \int_{J_1} h_\lambda \,
    \mathrm{d} \nu + \frac{C_1^\frac{1}{2}}{|\tilde{J}_2|^\frac{1}{2}}
    \int_{\tilde{J}_2} h_\lambda \, \mathrm{d} \nu \biggr)^2 \\ &=
    \frac{1}{\lambda} \frac{1}{|J_1|} \Biggl( \frac{ C_0^\frac{1}{2}
      \int_{J_1} h_\lambda \, \mathrm{d} \nu + C_1^\frac{1}{4}
      \frac{|J_1|^\frac{1}{2}}{|\tilde{J}_2|^\frac{1}{2}}
      \int_{\tilde{J}_2} h_\lambda \, \mathrm{d} \nu }{ \int_{J_1}
      h_\lambda \, \mathrm{d} \nu + \int_{\tilde{J}_2} h_\lambda \,
      \mathrm{d} \nu } \Biggr)^2 \biggl( \int_J h_\lambda \,
    \mathrm{d} \nu \biggr)^2. \label{eq:inductionstep}
  \end{align}
  We want to bound
  \[
    Q = \frac{ C_0^\frac{1}{2} \int_{J_1} h_\lambda \, \mathrm{d}
    \nu + C_1^\frac{1}{4}
    \frac{|J_1|^\frac{1}{2}}{|\tilde{J}_2|^\frac{1}{2}}
    \int_{\tilde{J}_2} h_\lambda \, \mathrm{d} \nu }{ \int_{J_1}
    h_\lambda \, \mathrm{d} \nu + \int_{\tilde{J}_2} h_\lambda \,
    \mathrm{d} \nu },
  \]
  and note that it is a weighted average of $C_0^\frac{1}{2}$ and
  $C_1^\frac{1}{2}
  \frac{|J_1|^\frac{1}{2}}{|\tilde{J}_2|^\frac{1}{2}}$. Let $d =
  |J_1|$ and $e = |\tilde{J}_2|^\frac{1}{2}$. Then $0 \leq e \leq
  d$. If we take $C_0$ much larger than $C_1$, we may conclude, by
  Lemma~\ref{lem:estimates_on_measure} and the fact that $Q$ is a
  weighted average, that
  \begin{align} \label{eq:Qest}
    \nonumber Q & \leq \frac{C_0^\frac{1}{2} c d^\theta +
      C_1^\frac{1}{2}
      \frac{|J_1|^\frac{1}{2}}{|\tilde{J}_2|^\frac{1}{2}} \frac{K}{2}
      (1 - \lambda)^{-\theta} e^\theta }{c d^\theta + \frac{K}{2} (1 -
      \lambda)^{-\theta} e^\theta} \\ &= \frac{C_0^\frac{1}{2} c (1 -
      \lambda)^\theta + C_1^\frac{1}{2} \frac{K}{2}
      \bigl(\frac{e}{d}\bigr)^{\theta-\frac{1}{2}} }{c (1 -
      \lambda)^\theta + \frac{K}{2} \bigl(\frac{e}{d} \bigr)^\theta }
    \leq C_0^\frac{1}{2} \eta,
  \end{align}
  where
  \[
  \eta = \sup_{0 \leq t \leq 1} \frac{2 c (1 - \lambda)^\theta + (C_1
    / C_0)^\frac{1}{2} K t^{\theta-\frac{1}{2}} }{2 c (1 -
    \lambda)^\theta + K t^\theta }.
  \]
 
  Combining \eqref{eq:inductionstep} and \eqref{eq:Qest}, we get that
  \begin{equation} \label{eq:induction3}
    \int_J h_\lambda^2 \, \mathrm{d} \nu \leq \frac{C_0 \eta^2}{|J|}
    \biggl( \int_J h_\lambda \, \mathrm{d} \nu \biggr)^2.
  \end{equation}

  Hence we have determined that if \eqref{eq:induction} holds for all
  intervals of a fixed size $L$, then \eqref{eq:induction3} holds for
  intervals of length $\lambda L$. By induction, starting with
  \eqref{eq:induction} for intervals of length $2 \rho$, we conclude
  that
  \begin{equation} \label{eq:inductionn}
    \int_J h_\lambda^2 \, \mathrm{d} \nu \leq \frac{C_0
      \eta^{2n}}{|J|} \biggl( \int_J h_\lambda \, \mathrm{d} \nu
    \biggr)^2
  \end{equation}
  holds for any interval of length $2 \lambda^n \rho$. This is not yet
  quite what we want. However, by choosing $C_0$ large, we can make
  $\eta$ arbitrarily close to $1$. In this way, for any $\varepsilon$,
  we will achieve the estimate
  \begin{equation*}
    \int_J h_\lambda^2 \, \mathrm{d} \nu \leq
    \frac{C_\varepsilon}{|J|^{1+\varepsilon}} \biggl( \int_J h_\lambda
    \, \mathrm{d} \nu \biggr)^2
  \end{equation*}
  for any interval of length $2 \lambda^n \rho$.  Since $C_0$ can be
  chosen to depend continuously on $\rho$, we conclude
  \eqref{eq:est_on_bernoulli} for intervals of any length.
\end{proof}

\subsection{Some estimates using Fourier analysis} \label{sec:fourier}

We let $\mu_{\lambda,k}$ be the measures defined in the previous
section. To emphasise the dependence on $\alpha$, which will prove
important in this section, we denote $\mu_{\lambda,k}$ by
$\mu_{\alpha,\lambda,k}$ and we let $h_{\alpha,\lambda,k}$ denote the
densities of the measures $\mu_{\alpha,\lambda,k}$. We are interested
in determining for which $\alpha$, $\lambda$ and $s$, there is a
constant $C$ such that
\[
\iint |x - y|^{-s} \, \mathrm{d}\mu_{\alpha,\lambda,k} (x)
\mathrm{d}\mu_{\alpha,\lambda,k} (y) < C,
\]
holds for infinitely many $k$. In this section, will prove the
following proposition.
\begin{proposition} \label{pro:energy_of_bernoulli}
  Let $\alpha s < 1$. If $\lambda \in (\frac{1}{2}, 0.64)$, then,
  almost surely, there exists a constant $C$ such that
  \begin{equation} \label{eq:energy_of_bernoulli} \iint |x-y|^{-s} \,
    \mathrm{d}\mu_{\alpha,\lambda,k} (x) \mathrm{d}
    \mu_{\alpha,\lambda,k} (y) \leq C,
  \end{equation}
  holds for infinitely many $k$.
\end{proposition}

Proposition~\ref{pro:energy_of_bernoulli} implies together with
Proposition~\ref{pro:est_on_bernoulli} the statement of
Theorem~\ref{the:bernoulli} for almost all $\lambda \in (\frac{1}{2},
0.64)$.

We will estimate the integrals in
Proposition~\ref{pro:energy_of_bernoulli} using Fourier transforms. We
use the convention that the Fourier transform of a function $f$ is the
function
\[
\hat{f} (\xi) = \int e^{-i2\pi\xi x} f(x) \, \mathrm{d} x.
\]

Writing as before, $R_s h (x) = |\cdot|^{-s} * h (x) = \int |x-y|^{-s}
h (y) \, \mathrm{d}y$, we have, using Parseval's formula, that
\begin{multline*}
  \iint |x - y|^{-s} \, \mathrm{d} \mu_{\alpha,\lambda,k} (x)
  \mathrm{d}\mu_{\alpha,\lambda,k} (y) = \int h_{\alpha,\lambda,k} (x)
  R_s h_{\alpha,\lambda,k} (x) \, \mathrm{d} x \\ = \int
  \overline{\hat{h}_{\alpha,\lambda,k} (\xi)} \widehat{R_s
    h_{\alpha,\lambda,k}} (\xi) \, \mathrm{d} \xi = c_s \int
  |\hat{h}_{\alpha,\lambda,k} (\xi)|^2 |\xi|^{s-1} \, \mathrm{d} \xi,
\end{multline*}
where $c_s$ is a constant depending only on $s$. We are going to
estimate $\int |\hat{h}_{\alpha,\lambda,k} (\xi)|^2 |\xi|^{s-1} \,
\mathrm{d} \xi$.

To determine the Fourier transform of $h_{\alpha,\lambda,k}$ we note
that the Fourier transform of the measure $\frac{1}{2} (\delta_a +
\delta_{0})$ is $e^{-i \pi a \xi} \cos (\pi a \xi)$. The measure
$\mu_{\alpha,\lambda,k}$ is the convolution of the measures
$\frac{1}{2} (\delta_{\lambda^n} + \delta_0)$, $n=0,1,\ldots, k$, and
the uniform mass-distribution on the interval $[-2^{-\alpha k},
  2^{-\alpha k}]$. Hence we have
\begin{equation} \label{eq:hproduct}
  \hat{h}_{\alpha,\lambda,k} (\xi) = \frac{\phi_{\lambda,k} (\xi)}{2
    \pi} \frac{\sin(2^{-\alpha k} \xi)}{2^{-\alpha k} \xi}
  \prod_{n=0}^k \cos (\pi \lambda^n \xi),
\end{equation}
and 
\[
\hat{h}_{\alpha,\lambda} (\xi) = \phi_{\lambda} (\xi)
\prod_{n=0}^\infty \cos (\pi \lambda^n \xi),
\]
where $|\phi_{\lambda,k} (\xi)| = |\phi_\lambda (\xi)| = 1$. We also
introduce the related function
\begin{equation} \label{eq:gproduct}
g_{\alpha, \lambda, k} (\xi) = \eta_{\alpha, k} (\xi) \prod_{n=0}^k
\cos (\pi \lambda^n \xi),
\end{equation}
where
\[
\eta_{\alpha, k} (\xi) = \left\{ \begin{array}{ll} 1 & \text{if }
  |\xi| \leq 2^{\alpha k}, \\ \frac{2^{\alpha k}}{\xi} & \text{if }
  |\xi| > 2^{\alpha k}. \end{array} \right.
\]

It appears from \eqref{eq:hproduct} and \eqref{eq:gproduct} that $2
\pi |\hat{h}_{\alpha,\lambda,k}| \leq |g_{\alpha, \lambda, k}|$.
Hence,
\[
\int |\hat{h}_{\alpha,\lambda,k} (\xi)|^2 |\xi|^{s-1} \, \mathrm{d}
\xi \leq \int |g_{\alpha,\lambda,k} (\xi)|^2 |\xi|^{s-1} \, \mathrm{d}
\xi = 2 \int_0^\infty |g_{\alpha,\lambda,k} (\xi)|^2 |\xi|^{s-1} \,
\mathrm{d} \xi,
\]
Moreover, instead of estimating $\int |g_{\alpha,\lambda,k} (\xi)|^2
|\xi|^{s-1} \, \mathrm{d} \xi$, we can do a bit more, and instead
estimate $\int |g_{\alpha,\lambda,k} (\xi)|^4 |\xi|^{2s-1} \,
\mathrm{d} \xi$.

\begin{proposition} \label{pro:fourierint}
  For almost all $\lambda \in [\frac{1}{2}, 0.64]$ there are constants
  $C$ and $D$ such that
  \[
  \int |g_{\alpha,\lambda,k} (\xi)|^4 |\xi|^{2s-1} \, \mathrm{d}\xi <
  C 4^{(\alpha s - 1) k} + D,
  \]
  holds for infinitely many $k$.
\end{proposition}

Since $|g_{\alpha,\lambda,k} (\xi)| \leq 1$, we can conclude from
Proposition~\ref{pro:fourierint}, that if $\alpha s < 1$, then for
almost all $\lambda \in [\frac{1}{2}, 0.64]$ there is a constant $C$
such that
\[
\int |\hat{h}_{\alpha,\lambda,k} (\xi)|^2 |\xi|^{s-1} \, \mathrm{d}\xi
\leq \int |g_{\alpha,\lambda,k} (\xi)|^2 |\xi|^{s-1} \, \mathrm{d}\xi
< C,
\]
holds for infinitely many $k$. Hence, Proposition~\ref{pro:fourierint}
implies Proposition~\ref{pro:energy_of_bernoulli}.

In fact, as we shall see in Section~\ref{sec:convolution},
Proposition~\ref{pro:fourierint} implies more, and will be important
to get the desired result not only for $\lambda$ in $(\frac{1}{2},
  0.64)$, but the entire interval $(\frac{1}{2}, 1)$.

\begin{remark} \label{rem:convolution}
  We have defined the measures $\mu_{\alpha, \lambda, k}$ so that
  their densities are normalised sums of indicator functions of
  intervals with radius $2^{-\alpha k}$ and centres in $F_{\lambda,
    k}$. Let $c > 0$. The proof of Proposition~\ref{pro:fourierint}
  will work without changes, if we had instead defined the measures
  $\mu_{\alpha, \lambda, k}$ such that their densities were based on
  intervals of radius $c 2^{-\alpha k}$ instead of $2^{-\alpha k}$.

  In Section~\ref{sec:convolution}, we will make use of this somewhat
  more general version of Proposition~\ref{pro:fourierint}, and we
  will then denote the corresponding measures by $\mu_{\alpha,
    \lambda, k, c}$, but to make the notation less heavy, we will
  only prove Proposition~\ref{pro:fourierint} in the case $c = 1$.
\end{remark}

\subsection{Proof of Proposition~\ref{pro:fourierint}} \label{sec:fourierproof}

Denote by $\lfloor N \rfloor$ the integer part of $N$, that is the
largest integer, not larger than $N$. We write the interval $[0,
  \infty)$ as the disjoint union $[0,\infty) = I_1 (k) \cup I_2 (k)
    \cup I_3 (k)$, where
\[
I_1 (k) = [0, 1), \quad I_2 (k) = [1, \lfloor 2^{\alpha k} \rfloor ),
    \quad \text{and} \quad I_3 (k) = [ \lfloor 2^{\alpha k} \rfloor,
      \infty),
\]
and treat separately the integrals
\[
J_i (\lambda) = \int_{I_i (k)} | g_{\alpha, \lambda, k} (\xi) |^4
|\xi|^{2s - 1} \, \mathrm{d} \xi, \quad i = 1, 2, 3.
\]

On the intervals $I_1 (k)$ and $I_2 (k)$ we have trivially that
\begin{equation} \label{eq:sinc1}
  \eta_{\alpha, k} (\xi)^4 = 1,
\end{equation}
and on the interval $I_3 (k)$, we have
\begin{equation} \label{eq:sinc2}
  \eta_{\alpha, k} (\xi)^4 \leq \frac{2^{4\alpha k}}{\xi^4}.
\end{equation}

Let us start by estimating $J_1 (\lambda)$. By \eqref{eq:sinc1} we get
that
\begin{equation} \label{eq:estJ1}
J_1 (\lambda) = \int_{I_1 (k)} |g_{\alpha,\lambda,k} (\xi)|^4
|\xi|^{2s-1} \, \mathrm{d} \xi \leq \int_{I_1 (k)}
|\xi|^{2s-1} \, \mathrm{d} \xi = \frac{1}{2 s} .
\end{equation}

Next, we estimate, using \eqref{eq:sinc1} and \eqref{eq:sinc2}, that
\begin{align*}
  J_2 (\lambda) = \int_{I_2 (k)} |g_{\alpha,\lambda,k} (\xi)|^4
  |\xi|^{2s-1} \, \mathrm{d} \xi & \leq \int_{1}^{\lfloor 2^{\alpha k}
    \rfloor} \biggl( \prod_{n = 0}^k \cos^2 (\lambda^{n} \xi)
  \biggr)^2 |\xi|^{2s-1} \, \mathrm{d} \xi, \\ J_3 (\lambda) =
  \int_{I_3 (k)} |g_{\alpha,\lambda,k} (\xi)|^4 |\xi|^{2s-1} \,
  \mathrm{d} \xi & \leq \int_{\lfloor 2^{\alpha k} \rfloor}^\infty
  \biggl( \prod_{n = 0}^k \cos^2 (\lambda^{n} \xi) \biggr)^2 2^{4
    \alpha k} |\xi|^{2s-5} \, \mathrm{d} \xi.
\end{align*}
We write
\[
P_k (\lambda, \xi) = \prod_{n=0}^k \cos^2 (\lambda^n \xi) =
\frac{1}{4^{k+1}} \sum_{a, b \in \Sigma_k} \cos ( \sum_{n=0}^k (a_n -
b_n) \lambda^n \xi),
\]
and put $\theta_{a, b} (\lambda) = \sum_{n=0}^k (a_n - b_n)
\lambda^n$.  Define $p_t \colon [1, \infty) \to \R$, such that $p_t
  (\xi) = n^{t}$ for $n \leq \xi < n+1$. Then, if $t < 0$, we have
  $p_t (\xi) \geq \xi^t$, and therefore
\begin{align*}
  J_2 (\lambda) &\leq \frac{1}{4^{k+1}} \sum_{a, b \in \Sigma_k}
  \int_{1}^{\lfloor 2^{\alpha k} \rfloor} \cos(\theta_{a, b} (\lambda)
  \xi) P_k (\lambda, \xi) p_{2s-1} (\xi) \, \mathrm{d}\xi, \\ J_3
  (\lambda) &\leq \frac{16^{\alpha k}}{4^{k+1}} \sum_{a, b \in
    \Sigma_k} \int_{\lfloor 2^{\alpha k} \rfloor}^\infty
  \cos(\theta_{a, b} (\lambda) \xi) P_k (\lambda, \xi) p_{2s-5} (\xi)
  \, \mathrm{d}\xi.
\end{align*}
We will now chop up the integrals above into sums of integrals over
intervals $[m, m+1]$. The number $2^{\alpha k}$ is not necessarily an
integer so it may not hold that $2^{\alpha k} = \lfloor 2^{\alpha k}
\rfloor$. However, to make the notation lighter, we shall adopt the
convention to write $\sum^{2^{\alpha k}}$ instead of $\sum^{\lfloor
  2^{\alpha k} \rfloor}$, and $\sum_{2^{\alpha k}}$ instead of
$\sum_{\lfloor 2^{\alpha k} \rfloor}$.  Hence we write
\begin{align*}
  J_2 (\lambda) &\leq \frac{1}{4^{k+1}} \sum_{a, b \in \Sigma_k}
  \sum_{m = 1}^{2^{\alpha k}} \int_m^{m+1} \cos(\theta_{a, b} (\lambda)
  \xi) P_k (\lambda, \xi) m^{2s-1} \, \mathrm{d}\xi, \\ J_3 (\lambda)
  &\leq \frac{16^{\alpha k}}{4^{k+1}} \sum_{a, b \in \Sigma_k} \sum_{m
    = 2^{\alpha k}}^\infty \int_m^{m+1} \cos(\theta_{a, b} (\lambda) \xi)
  P_k (\lambda, \xi) m^{2s-5} \, \mathrm{d}\xi.
\end{align*}
If $a$ and $b$ are two different elements in $\Sigma_k$, then
$\theta_{a, b} (\lambda) \neq 0$, except for finitely many
$\lambda$. Therefore, for $a \neq b$, and almost all $\lambda$, we
have
\begin{multline*}
  \int_m^{m+1} \cos(\theta_{a, b} (\lambda) \xi) P_k (\lambda, \xi)
  m^{2s-1} \, \mathrm{d}\xi \leq \int_m^{m+1} \cos(\theta_{a, b} (\lambda)
  \xi) m^{2s-1} \, \mathrm{d}\xi \\ = \biggl( \frac{\sin (\theta_{a, b}
    (\lambda) (m+1) )}{\theta_{a, b} (\lambda)} - \frac{\sin (\theta_{a, b}
    (\lambda) m )}{\theta_{a, b} (\lambda)} \biggr) m^{2s-1},
\end{multline*}
and we can thus write
\begin{align*}
  \sum_{m = 1}^{2^{\alpha k}} & \int_{m}^{m+1} \cos (\theta_{a, b}
  (\lambda) \xi ) P_k (\lambda, \xi ) m^{2s-1} \, \mathrm{d} \xi
  \mathrm{d} \lambda \\ \leq & \sum_{m = 1}^{2^{\alpha k}}
  \int_{m}^{m+1} \cos (\theta_{a, b} (\lambda) \xi ) m^{2s-1} \, \mathrm{d}
  \xi \mathrm{d} \lambda\\ =& \sum_{m = 2}^{1 + 2^{\alpha k}}
  \frac{\sin(\theta_{a, b} (\lambda) m)}{\theta_{a, b} (\lambda)} (m-1)^{2s-1} -
  \sum_{m = 1}^{2^{\alpha k}} \frac{\sin(\theta_{a, b} (\lambda) m)}{g_{a,
      b} (\lambda)} m^{2s-1} \\ =& \frac{\sin (\theta_{a, b} (\lambda) (1 +
    2^{\alpha k})}{\theta_{a, b} (\lambda)} (2^{\alpha k})^{2s - 1} -
  \frac{\sin (\theta_{a, b} (\lambda)}{\theta_{a, b} (\lambda)} \\ &+ \sum_{m =
    2}^{2^{\alpha k}} \frac{\sin (\theta_{a, b} (\lambda) m}{\theta_{a, b}
    (\lambda)} \bigl( (m - 1)^{2s - 1} - m^{2s - 1} \bigr).
\end{align*}

We now consider $[p,q] \subset (\frac{1}{2}, 0.64)$, and estimate
$\int_p^q J_2 (\lambda) \, \mathrm{d}\lambda$. For this purpose we
will use the following lemma. To state it, we use the notation $l
(a,b)$ to denote the smallest integer $l$ such that $a_l \neq b_l$ if
$a$ and $b$ are two different elements of $\Sigma_k$. We also put
\[
J_{a,b,m} = \int_p^q \frac{\sin (\theta_{a, b} (\lambda)
    m)}{\theta_{a, b} (\lambda)} \, \mathrm{d} \lambda.
\]

\begin{lemma} \label{lem:transint}
  There is a constant $K_1$ such that for all $a,b \in \Sigma_k$, with
  $a \neq b$, and all $m$,
  \[
  |J_{a,b,m}| \leq K_1 p^{-l (a,b)}.
  \]
\end{lemma}

It is intuitively clear that Lemma~\ref{lem:transint} follows from
Solomyak's transversality lemma \cite{solomyak}. Details on how to
prove this, are available in \cite{solomyak}, where it is shown how it
follows from a lemma in \cite{palis-yoccoz}, that is called Lemma~2.2
in \cite{solomyak}.

By a change of order of integration we have that
\[
\int_p^q J_2 (\lambda) \, \mathrm{d} \lambda \leq L_1 + L_2,
\]
where
\begin{align*}
  L_1 &= \frac{1}{4^{k+1}} \sum_{l = 0}^{k} \sum_{\substack{a,b \in
      \Sigma_{k,l} \\ l(a,b) = l}} \biggl(\sum_{m = 2}^{2^{\alpha k}}
  J_{a,b,m} \bigl( (m - 1)^{2s - 1} - m^{2s - 1} \bigr)
  \\ & \hspace{4cm} + J_{a,b,1+2^{\alpha k}} (2^{\alpha k})^{2s-1} -
  J_{a,b,1} \biggr), \\ L_2 &= \frac{1}{4^{k+1}} \sum_{a = b \in
    \Sigma_k } \sum_{m = 1}^{2^{\alpha k}} \int_p^q \int_m^{m + 1}
  \cos(\theta_{a, b} (\lambda) \xi) P_k (\lambda, \xi) m^{2s-1} \,
  \mathrm{d}\xi \mathrm{d}\lambda.
\end{align*}
The first part is estimated with use of Lemma~\ref{lem:transint}. We
have
\begin{align*}
  L_1 &\leq \frac{1}{4^{k+1}} \sum_{l = 0}^{k} \sum_{a,b \in
    \Sigma_{k,l} } \frac{K_1}{p^{l}} \biggl( \sum_{m = 2}^{2^{\alpha
      k}} \bigl( (m - 1)^{2s - 1} - m^{2s - 1} \bigr) + (2^{\alpha
    k})^{2s-1} + 1 \biggr), \\ & = \frac{1}{4^{k+1}} \sum_{l = 0}^k
  2^{k+1} 2^{k+1-l} 2 K_1 p^{-l} \leq \frac{2 K_1}{1 - (2p)^{-1}}.
\end{align*}

We now turn to the estimate of $L_2$. If $a = b$, then $\theta_{a, b}
= 0$, and so
\begin{align*}
L_2 &= \frac{1}{2^{k+1}} \sum_{m = 1}^{2^{\alpha k}} \int_p^q
\int_m^{m + 1} P_k (\lambda, \xi) m^{2s-1} \, \mathrm{d}\xi
\mathrm{d}\lambda \\ &= \frac{1}{6^{k+1}} \sum_{c,d \in \Sigma_k}
\sum_{m=1}^{2^{\alpha k}} \int_p^q \int_m^{m+1} \cos(\theta_{c,d} (\lambda)
\xi) m^{2s - 1} \, \mathrm{d}\xi \mathrm{d}\lambda
\\ &=\frac{1}{2^{k+1}} L_1 + \frac{1}{6^{k+1}} \sum_{c=d \in \Sigma_k}
\sum_{m=1}^{2^{\alpha k}} \int_p^q \int_m^{m+1} \cos(\theta_{c,d} (\lambda)
\xi) m^{2s - 1} \, \mathrm{d}\xi \mathrm{d} \lambda \\ &\leq
\frac{1}{2^{k+1}} L_1 + \frac{1}{6^{k+1}} 2^{k+1}
\sum_{m=1}^{2^{\alpha k}} (q-p) m^{2s - 1} \\ &\leq \frac{1}{2^{k+1}}
L_1 + \frac{1}{4^{k+1}} K_2 2^{2 \alpha s k},
\end{align*}
where $K_2$ is a constant that does not depend on $k$.

Putting the estimates of $L_1$ and $L_2$ together, we find that
\begin{equation} \label{eq:estJ2}
  \int_p^q J_2 ( \lambda) \, \mathrm{d} \lambda \leq \frac{3 K_1}{1 -
    (2p)^{-1}} + \frac{K_2}{4} 2^{(2 \alpha s - 2) k} \leq K_3 (1 + 2^{(2
    \alpha s - 2) k}),
\end{equation}
where $K_3$ is a constant that does not depend on $k$.

We will now estimate $J_3 (\lambda)$. In the same way as for $J_2$, we
have that
\[
\int_p^q J_3 (\lambda) \, \mathrm{d} \lambda \leq M_1 + M_2,
\]
where
\begin{align*}
  M_1 &= \frac{16^{\alpha k}}{4^{k+1}} \sum_{l = 0}^{k}
  \sum_{\substack{a,b \in \Sigma_{k,l} \\ l(a,b) = l}} \biggl(\sum_{m
    = 2^{\alpha k}}^\infty J_{a,b,m} \bigl( (m - 1)^{2s - 5} - m^{2s -
    5} \bigr) \\ & \hspace{4cm} + J_{a,b,1+2^{\alpha k}} (2^{\alpha
    k})^{2s-5} - J_{a,b,1} \biggr), \\ M_2 &= \frac{16^{\alpha
      k}}{4^{k+1}} \sum_{a = b \in \Sigma_k} \sum_{m = 2^{\alpha
      k}}^\infty \int_p^q \int_m^{m + 1} \cos(\theta_{a, b} (\lambda) \xi)
  P_k (\lambda, \xi) m^{2s-5} \, \mathrm{d}\xi \mathrm{d}\lambda.
\end{align*}
The first part is again estimated with use of
Lemma~\ref{lem:transint}. We have
\begin{align*}
  M_1 &\leq \frac{16^{\alpha k}}{4^{k+1}} \sum_{l = 0}^{k} \sum_{a,b
    \in \Sigma_{k,l}} \frac{K_1}{p^{l}} \biggl( \sum_{m = 2^{\alpha k}}^\infty
    \bigl( (m - 1)^{2s - 5} - m^{2s - 5} \bigr) + (2^{\alpha
      k})^{2s-5} + 1 \biggr), \\ & = \frac{16^{\alpha k}}{4^{k+1}}
    \sum_{l = 0}^k 2^{k+1} 2^{k+1-l} 2 K_1 p^{-l} \leq \frac{2
      K_1}{1 - (2p)^{-1}}.
\end{align*}

We proceed with the estimate of $M_2$. If $a = b$, then $\theta_{a, b}
= 0$, and so
\begin{align*}
M_2 &= \frac{16^{\alpha k}}{2^{k+1}} \sum_{m = 2^{\alpha k}}^\infty
\int_p^q \int_m^{m+1} P_k (\lambda, \xi) m^{2s - 5} \, \mathrm{d}\xi
\mathrm{d} \lambda \\ &= \frac{16^{\alpha k}}{6^{k + 1}} \sum_{c, d
  \in \Sigma_k} \sum_{m=2^{\alpha k}}^\infty \int_p^q \int_m^{m+1}
\cos(\theta_{c,d} (\lambda) \xi) m^{2s - 5} \, \mathrm{d} \xi \mathrm{d}
\lambda \\ &= \frac{1}{2^{k+1}} M_1 + \frac{16^{\alpha k}}{6^{k + 1}}
\sum_{c = d \in \Sigma_k} \sum_{m = 2^{\alpha k}}^\infty \int_p^q
\int_m^{m+1} \cos( \theta_{c,d} (\lambda) \xi) m^{2s - 5} \, \mathrm{d} \xi
\mathrm{d} \lambda \\ &\leq \frac{1}{2^{k+1}} M_1 + \frac{16^{\alpha
    k}}{6^{k+1}} 2^{k+1} \sum_{m=2^{\alpha k}}^\infty (q-p) m^{2s - 5}
\\ &\leq \frac{1}{2^{k + 1}} M_1 + \frac{1}{4^{k+1}} K_4 2^{2 \alpha s
  k}.
\end{align*}
where $K_4$ does not depend on $k$. 

The estimates of $M_1$ and $M_2$ imply that
\begin{equation} \label{eq:estJ3}
  \int_p^q J_3 ( \lambda) \, \mathrm{d} \lambda \leq \frac{3 K_1}{1 -
    (2p)^{-1}} + K_4 2^{(2 s \alpha - 2) k} \leq K_5 (1 + 2^{(2 \alpha
    s - 2) k}),
\end{equation}
where $K_5$ is a constant.

From \eqref{eq:estJ1}, \eqref{eq:estJ2} and \eqref{eq:estJ3}, we
conclude that
\[
\int_p^q \int |\hat{h}_{\lambda, k} (\xi)|^4 |\xi|^{2s - 1} \,
\mathrm{d} \xi \mathrm{d} \lambda \leq \frac{1}{s} + 2 (K_3 +
K_5) (1 + 2^{(2\alpha s - 2)k}).
\]
Hence, for almost all $\lambda \in [p,q]$, there are constants $C =
C(\lambda)$ and $D = D(\lambda)$ such that
\[
\int |\hat{h}_{\lambda, k} (\xi)|^4 |\xi|^{2s - 1} \, \mathrm{d} \xi <
C 2^{2(\alpha s - 1) k} + D,
\]
holds for infinitely many $k$. Since $p$ and $q$ are arbitrary this
proves Proposition~\ref{pro:fourierint}.

\subsection{Convolutions} \label{sec:convolution}

We have proved the statement of Theorem~\ref{the:bernoulli} for almost
all $\lambda \in (\frac{1}{2}, 0.64)$. In this section we are going to
show the result for almost all $\lambda \in (\frac{1}{2}, 1)$. 

Let $\lambda$ be such that $\lambda^2 \in (\frac{1}{2}, 0.64)$. We
define the measure $\mu_{\alpha, \lambda, 2k+1}^{(2)}$ as the
convolution of $\mu_{2\alpha, \lambda^2, k, c}$ and $\mu_{2 \alpha,
  \lambda^2, k, c} \circ S_1^{-1}$, where $S_1$ is the contraction
$S_1 (x) = \lambda x$. Hence, if we let $h_{\alpha, \lambda,
  2k+1}^{(2)}$ denote the density of $\mu_{\alpha, \lambda,
  2k+1}^{(2)}$, we have $h_{\alpha, \lambda, 2k+1}^{(2)} = h_{2\alpha,
  \lambda^2, k, c} * (h_{2\alpha, \lambda^2, k, c} \circ S_1^{-1})$.

It follows, if we choose the constant $c$ appropriately, that the
measure $\mu_{\alpha, \lambda, 2k+1}^{(2)}$ is absolutely continuous
with respect to $\mu_{\alpha, \lambda, 2k+1}$, that is, the support of
$\mu_{\alpha, \lambda, 2k+1}^{(2)}$ is in $E_{\lambda, 2k + 1}
(\alpha)$. This makes it natural to try to apply
Theorem~\ref{the:frostman} to the measures $\mu_{\alpha, \lambda,
  2k+1}^{(2)}$. Moreover, it is not difficult to see that
$\mu_{\alpha, \lambda, 2k+1}^{(2)}$ converges weakly to $\mu_\lambda$,
the distribution of the corresponding Bernoulli convolution, as $k \to
\infty$.

We will now prove the following result for the measures $\mu_{\alpha,
  \lambda, 2k+1}^{(2)}$, analogous to
Proposition~\ref{pro:energy_of_bernoulli}.

\begin{proposition} \label{pro:convolution}
  Let $\alpha s < 1$. If $\lambda^2 \in (\frac{1}{2}, 0.64)$, then,
  almost surely, there exists a constant $C$ such that
  \[
  \iint |x - y|^{-s} \, \mathrm{d} \mu_{\alpha, \lambda, 2k+1}^{(2)}
  (x) \mathrm{d} \mu_{\alpha, \lambda, 2k+1}^{(2)} (y) \leq C,
  \]
  holds for infinitely many $k$.
\end{proposition}

Proposition~\ref{pro:convolution}, together with
Proposition~\ref{pro:est_on_bernoulli}, implies the statement of
Theorem~\ref{the:bernoulli} for almost all $\lambda$ such that
$\lambda \in (\frac{1}{2}, 0.64)$ or $\lambda^2 \in (\frac{1}{2},
0.64)$, that is for almost all $\lambda \in (\frac{1}{2}, \frac{4}{5})
= (\frac{1}{2}, \sqrt{0.64})$.

\begin{proof}[Proof of Proposition~\ref{pro:convolution}]
  To prove Proposition~\ref{pro:convolution}, we do as in
  Section~\ref{sec:fourier}, and write
  \[
  \iint |x - y|^{-s} \, \mathrm{d} \mu_{\alpha,\lambda,2k+1}^{(2)} (x)
  \mathrm{d}\mu_{\alpha,\lambda, 2k+1}^{(2)} (y) = c_s \int
  |\hat{h}_{\alpha,\lambda,2k+1}^{(2)} (\xi)|^2 |\xi|^{s-1} \,
  \mathrm{d} \xi.
  \]
  Now, we use the fact that $h_{\alpha, \lambda, 2k+1}^{(2)} =
  h_{2\alpha, \lambda^2, k, c} * (h_{2\alpha, \lambda^2, k, c} \circ
  S_1^{-1})$, or equivalently $\hat{h}_{\alpha, \lambda, 2k+1}^{(2)} =
  \hat{h}_{2\alpha, \lambda^2, k, c} \cdot \widehat{(h_{2\alpha,
      \lambda^2, k, c} \circ S_1^{-1})}$, together with the
  Cauchy--Bunyakovsky--Schwarz inequality to conclude that
  \[
  \iint |x - y|^{-s} \, \mathrm{d} \mu_{\alpha,\lambda,2k+1}^{(2)} (x)
  \mathrm{d}\mu_{\alpha,\lambda, 2k+1}^{(2)} (y) \leq c_{s,\lambda}
  \int |\hat{h}_{2 \alpha, \lambda^2, k} (\xi)|^4 |\xi|^{s-1} \,
  \mathrm{d} \xi,
  \]
  where $c_{s,\lambda}$ is a constant that only depends on $s$ and
  $\lambda$. By Proposition~\ref{pro:fourierint} we now get that
  \[
  \iint |x - y|^{-s} \, \mathrm{d} \mu_{\alpha,\lambda,2k+1}^{(2)} (x)
  \mathrm{d}\mu_{\alpha,\lambda, 2k+1}^{(2)} (y) \leq c_{s,\lambda} (C
  4^{(\alpha s - 1) k} + D).
  \]
  This clearly implies Proposition~\ref{pro:convolution}.
\end{proof}

We can now consider higher powers of convolutions of scalings of the
measures $\mu_{\alpha, \lambda, k}$. Similarly as was done in
Proposition~\ref{pro:convolution}, we can conclude the statement of
Theorem~\ref{the:bernoulli} for almost all $\lambda \in (\frac{1}{2},
\sqrt[m]{\frac{4}{5}} ) = (\frac{1}{2}, \sqrt[m+1]{0.64})$, if we
prove, instead of Proposition~\ref{pro:fourierint}, an estimate on
\[
\int |g_{\alpha,\lambda,k} (\xi)|^{2^{m+2}} |\xi|^{2^{m+1} s - 1} \,
\mathrm{d}\xi,
\]
analogous to Proposition~\ref{pro:fourierint}. This can be done for
any $m$ in a straight-forward way, similar to the proof of
Proposition~\ref{pro:fourierint}, but would be somewhat lengthy and
cumbersome. We will therefore leave out the details, since the proof
of Proposition~\ref{pro:fourierint} contains all the necessary
ideas. Since $\sqrt[m]{4/5} \to 1$ as $m \to \infty$, this concludes
the proof of Theorem~\ref{the:bernoulli}.


\begin{thebibliography}{1}

\bibitem{bugeaud1} Y. Bugeaud, {\em Intersective sets and Diophantine
  approximation}, Michigan Mathematical Journal 52 (2004), 667--682.

\bibitem{falconerold} K. Falconer, {\em Classes of sets with large
  intersections}, Mathematika 32 (1985), no. 2, 191--205.

\bibitem{falconernew} K. Falconer, {\em Sets with large intersection
  properties}, Journal of the London Mathematical Society 49 (1994),
  no. 2, 267--280.

\bibitem{palis-yoccoz} J. Palis, J.-C. Yoccoz, {\em Homoclinic
  tangencies for hyperbolic sets of large Hausdorff dimension}, Acta
  Mathematica 172 (1994), 91–136.

\bibitem{persson-reeve} T. Persson, H. Reeve, {\em On the diophantine
  properties of $\lambda$-expansions}, Mathematika 59 (2013), no. 1,
  65--86.

\bibitem{solomyak} B. Solomyak, {\em On the random series $\sum \pm
  \lambda^n$ (an Erd\H{o}s problem)}, Annals of Mathematics 142
  (1995), no. 3, 611--625.

\end{thebibliography}
\end{document}